\documentclass[12pt,leqno]{article}
\usepackage{amsfonts}
\usepackage{amsmath, dsfont, amsthm, amsfonts, amssymb, color}
\usepackage{mathrsfs}
\usepackage{color}
\usepackage{stmaryrd}
\newtheorem{thm}{Theorem}[section]

\newtheorem{lem}[thm]{Lemma}

\newtheorem{rem}[thm]{Remark}
\theoremstyle{definition}

\newcommand{\scr}[1]{\mathscr #1}
\definecolor{wco}{rgb}{0.5,0.2,0.3}

\numberwithin{equation}{section} \theoremstyle{remark}

\newcommand{\ua}{\uparrow}

\title{{\bf
 Quantitative Propagation of Chaos in $L^\eta(\eta\in(0,1])$-Wasserstein distance for Mean Field Interacting Particle System}\footnote{Supported in
 part by  National Key R\&D Program of China (No. 2022YFA1006000) and NNSFC (12271398).} }
\author{
{\bf   Xing Huang  }\\
\footnotesize{ Center for Applied Mathematics, Tianjin
University, Tianjin 300072, China}\\
\footnotesize{  xinghuang@tju.edu.cn}}
\begin{document}
\allowdisplaybreaks
\def\R{\mathbb R}  \def\ff{\frac} \def\ss{\sqrt} \def\B{\mathbf
B} \def\W{\mathbb W}
\def\N{\mathbb N} \def\kk{\kappa} \def\m{{\bf m}}
\def\ee{\varepsilon}\def\ddd{D^*}
\def\dd{\delta} \def\DD{\Delta} \def\vv{\varepsilon} \def\rr{\rho}
\def\<{\langle} \def\>{\rangle} \def\GG{\Gamma} \def\gg{\gamma}
  \def\nn{\nabla} \def\pp{\partial} \def\E{\mathbb E}
\def\d{\text{\rm{d}}} \def\bb{\beta} \def\aa{\alpha} \def\D{\scr D}
  \def\si{\sigma} \def\ess{\text{\rm{ess}}}
\def\beg{\begin} \def\beq{\begin{equation}}  \def\F{\scr F}
\def\Ric{\text{\rm{Ric}}} \def\Hess{\text{\rm{Hess}}}
\def\e{\text{\rm{e}}} \def\ua{\underline a} \def\OO{\Omega}  \def\oo{\omega}
 \def\tt{\tilde} \def\Ric{\text{\rm{Ric}}}
\def\cut{\text{\rm{cut}}} \def\P{\mathbb P} \def\ifn{I_n(f^{\bigotimes n})}
\def\C{\scr C}      \def\aaa{\mathbf{r}}     \def\r{r}
\def\gap{\text{\rm{gap}}} \def\prr{\pi_{{\bf m},\varrho}}  \def\r{\mathbf r}
\def\Z{\mathbb Z} \def\vrr{\varrho}
\def\L{\scr L}\def\Tt{\tt} \def\TT{\tt}\def\II{\mathbb I}
\def\i{{\rm in}}\def\Sect{{\rm Sect}}  \def\H{\mathbb H}
\def\M{\scr M}\def\Q{\mathbb Q} \def\texto{\text{o}}
\def\Rank{{\rm Rank}} \def\B{\scr B} \def\i{{\rm i}} \def\HR{\hat{\R}^d}
\def\to{\rightarrow}\def\l{\ell}\def\iint{\int}
\def\EE{\scr E}\def\Cut{{\rm Cut}}
\def\A{\scr A} \def\Lip{{\rm Lip}}
\def\BB{\scr B}\def\Ent{{\rm Ent}}\def\L{\scr L}
\def\R{\mathbb R}  \def\ff{\frac} \def\ss{\sqrt} \def\B{\mathbf
B}
\def\N{\mathbb N} \def\kk{\kappa} \def\m{{\bf m}}
\def\dd{\delta} \def\DD{\Delta} \def\vv{\varepsilon} \def\rr{\rho}
\def\<{\langle} \def\>{\rangle} \def\GG{\Gamma} \def\gg{\gamma}
  \def\nn{\nabla} \def\pp{\partial} \def\E{\mathbb E}
\def\d{\text{\rm{d}}} \def\bb{\beta} \def\aa{\alpha} \def\D{\scr D}
  \def\si{\sigma} \def\ess{\text{\rm{ess}}}
\def\beg{\begin} \def\beq{\begin{equation}}  \def\F{\scr F}
\def\Ric{\text{\rm{Ric}}} \def\Hess{\text{\rm{Hess}}}
\def\e{\text{\rm{e}}} \def\ua{\underline a} \def\OO{\Omega}  \def\oo{\omega}
 \def\tt{\tilde} \def\Ric{\text{\rm{Ric}}}
\def\cut{\text{\rm{cut}}} \def\P{\mathbb P} \def\ifn{I_n(f^{\bigotimes n})}
\def\C{\scr C}      \def\aaa{\mathbf{r}}     \def\r{r}
\def\gap{\text{\rm{gap}}} \def\prr{\pi_{{\bf m},\varrho}}  \def\r{\mathbf r}
\def\Z{\mathbb Z} \def\vrr{\varrho}
\def\L{\scr L}\def\Tt{\tt} \def\TT{\tt}\def\II{\mathbb I}
\def\i{{\rm in}}\def\Sect{{\rm Sect}}  \def\H{\mathbb H}
\def\M{\scr M}\def\Q{\mathbb Q} \def\texto{\text{o}} \def\LL{\Lambda}
\def\Rank{{\rm Rank}} \def\B{\scr B} \def\i{{\rm i}} \def\HR{\hat{\R}^d}
\def\to{\rightarrow}\def\l{\ell}
\def\8{\infty}\def\I{1}\def\U{\scr U} \def\n{{\mathbf n}}
\maketitle

\begin{abstract} In this paper, quantitative propagation of chaos in $L^\eta$($\eta\in(0,1]$)-Wasserstein distance for mean field interacting particle system is derived, where the diffusion coefficient is allowed to be interacting and the initial distribution of interacting particle system converges to that of the limit equation in $L^1$-Wasserstein distance.  The non-degenerate and second order system are investigated respectively and the main tool relies on the gradient estimate of the decoupled SDEs.
 \end{abstract}

 \vskip 2cm
\section{Introduction}
In the mean field interacting particle system, where the coefficients depend on the empirical distribution of the particles, as the number of particles goes to infinity, the limit equation of a single particle is distribution dependent stochastic differential equation(SDE), which is also called McKean-Vlasov SDE in the literature due to the work in \cite{McKean}, one can refer to the recent monograph \cite{WR2024} for plentiful contributions on this type SDE. This aforementioned limit phenomenon is related to propagation of chaos, which can be viewed as a dynamical version of Kac's chaotic property, which is called the Boltzmann property in \cite{Kac}.

Let $(E,\rho)$ be a Polish space and $o\in E$ be a fixed point. Let $\scr P(E)$ be the set of all probability measures on $E$ equipped with the weak topology. For $\eta>0$, let
$$\scr P_\eta(E):=\big\{\mu\in \scr P(E): \mu(\rho(o,\cdot)^\eta)<\infty\big\},$$
which is a Polish space under the $L^\eta$-Wasserstein distance
$$\W_\eta(\mu,\nu)= \inf_{\pi\in \mathbf{C}(\mu,\nu)}\left\{\int_{E\times E} \rho(x,y)^\eta \pi(\d x,\d y)\right\}^{\frac{1}{1\vee \eta}},\ \  \mu,\nu\in \scr P_\eta(E),$$ where $\mathbf{C}(\mu,\nu)$ is the set of all couplings of $\mu$ and $\nu$.

When $E=\R^d,\eta\in(0,1]$,
  the following dual formula holds:
$$\W_\eta(\gamma,\tilde{\gamma})=\sup_{[f]_{\eta}\leq 1}|\gamma(f)-\tilde{\gamma}(f)|,\ \ \gamma,\tilde{\gamma}\in \scr P_\eta(\R^d),$$
where $[f]_\eta:=\sup_{x\neq y}\frac{|f(x)-f(y)|}{|x-y|^\eta}$, see for instance \cite[Theorem 5.10]{Chen04}.

In the mean filed interacting particle system, we usually need to consider finitely many particles so that it is natural to consider the distance on $\scr P_\eta((\R^d)^m)$ for any $m\geq 2$. To this end, for any $m\geq 2$ and $\eta\in(0,1]$, we adopt the following distance on $(\R^d)^m$:
$$\|x-y\|_{1,\eta}=\sum_{i=1}^m|x^i-y^i|^\eta,\ \ x=(x^1,x^2,\cdots,x^m), y=(y^1,y^2,\cdots,y^m)\in(\R^d)^m.$$
Define the associated Wasserstein distance:
\begin{align}\label{dew}\tilde{\W}_\eta(\mu,\nu)=\inf_{\pi\in\mathbf{C}(\mu,\nu)}\int_{(\R^d)^m\times (\R^d)^m}\|x-y\|_{1,\eta}\pi(\d x,\d y),\ \ \mu,\nu\in \scr P_\eta((\R^d)^m).
\end{align}
In this case, the dual formula becomes
\begin{align}\label{dfa}\tilde{\W}_\eta(\gamma,\tilde{\gamma})=\sup_{[f]_{1,\eta}\leq 1}|\gamma(f)-\tilde{\gamma}(f)|,\ \ \gamma,\tilde{\gamma}\in \scr P_\eta((\R^d)^m),
\end{align}
where \begin{align}\label{1et}[f]_{1,\eta}:=\sup_{x\neq y}\frac{|f(x)-f(y)|}{\|x-y\|_{1,\eta}}.\end{align}
For any $k\geq 1, m\geq 1$, let
$$ C_b^k((\R^d)^m):= \big\{f: (\R^d)^m\to  \R \  \text{has bounded and continuous up to $k$ order derivatives}\big\}. $$
 Noting that for any $\gamma,\tilde{\gamma}\in \scr P_\eta((\R^d)^m)$, $C_b^2((\R^d)^m)\cap\{f:(\R^d)^m\to\R, [f]_{1,\eta}\leq 1\}$ is dense in $\{f:(\R^d)^m\to\R, [f]_{1,\eta}\leq 1\}$ under $L^1(\gamma+\tilde{\gamma})$, it holds
\begin{align}\label{dfa1}\tilde{\W}_\eta(\gamma,\tilde{\gamma})=\sup_{f\in C_b^2((\R^d)^m), [f]_{1,\eta}\leq 1}|\gamma(f)-\tilde{\gamma}(f)|,\ \ \gamma,\tilde{\gamma}\in \scr P_\eta((\R^d)^m).
\end{align}
To introduce the mean field interacting particle system, let $T>0$, $W_t$ be an $n$-dimensional Brownian motion on some complete filtration probability space $(\Omega, \scr F, (\scr F_t)_{\in[0,T]},\P)$. Let $X_0$ be an $\F_0$-measurable random variable,
$N\ge1$ be an integer and $(X_0^i,W^i_t)_{1\le i\le N}$ be i.i.d.\,copies of $(X_0,W_t).$
Let $b:[0,T]\times \R^d\times\scr P(\R^d)\to\R^d$ and $\sigma:[0,T]\times \R^d\times\scr P(\R^d)\to\R^d\otimes\R^{n}$ be measurable and bounded on bounded set. Consider the non-interacting particle system, which solves independent McKean-Vlasov SDEs:
\begin{align}\label{MVD}\d X_t^i= b_t(X_t^i, \L_{X_t^i})\d t+  \sigma_t(X^i_t,\L_{X_t^i}) \d W^i_t,\ \ 1\leq i\leq N
\end{align}
for $\L_{X_t^i}$ being the distribution of $X_t^i$. The associated mean field interacting particle system is
\begin{align}\label{mfi}
\d X^{i,N}_t=b_t(X_t^{i,N}, \hat\mu_t^N)\d t+\sigma_t(X^{i,N}_t, \hat\mu_t^N) \d W^i_t,\ \ 1\leq i\leq N,
\end{align}
where $\hat\mu_t^N$ is the empirical distribution of $(X_t^{i,N})_{1\leq i\leq N}$, i.e.
\begin{equation*}
 \hat\mu_t^N =\ff{1}{N}\sum_{j=1}^N\dd_{X_t^{j,N}}.
 \end{equation*}
Throughout the paper, we assume that the distribution of $(X_0^{i,N})_{1\leq i\leq N}$ is exchangeable. When \eqref{MVD} and \eqref{mfi} are well-posed, for any $\mu\in\scr P(\R^d)$, let $P_t^\ast\mu$ be the distribution of $X_t^1$ initial distribution $\mu$, and for any exchangeable $\mu^N\in\scr P((\R^d)^N)$, $1\leq k\leq N$, $(P_t^{[k],N})^\ast\mu^N$
be the distribution of $(X_t^{i,N})_{1\leq i\leq k}$ with initial distribution $\mu^N$. Moreover, let $\mu^{\otimes k}$ denote the $k$ independent product of $\mu$, i.e. $\mu^{\otimes k}=\prod_{i=1}^k\mu$.

We first recall some recent results on propagation of chaos for mean field interacting particle system. In \cite{BJW,JW,JW1}, the authors adopt the entropy method to derive the quantitative entropy-entropy type propagation of chaos for mean field particle system with additive noise and singular interaction kernel. Also in the additive noise case, the author of \cite{L21} obtains the sharp rate of entropy-entropy type propagation of chaos for particle system with bounded or Lipschitz continuous interaction kernel by the technique of BBGKY hierarchy.
 \cite{STW,SW} derive the quantitative propagation of chaos in $\W_2$-distance of mean field interacting particle system with regime-switching. \cite{CLY} considers the conditional propagation of chaos in $\W_p(p\geq 2)$-distance for mean field interacting particle system with common noise. For mean field interacting particle system with feedback control, \cite{WHGY} acquires the propagation of chaos in $\W_2$-distance. \cite{NW} proves the propagation of chaos in strong sense for mean field interacting particle system in the frame of stochastic variational inequalities, where the drift is of super-linear
growth and locally Lipschitz continuous and the diffusion is locally H\"{o}lder continuous. In \cite{HX23e}, the author proposes and obtains  the entropy-cost type propagation of chaos and the initial distribution of interacting particle system is allowed to be singular with that of the limit equation. For the uniform in time propagation of chaos in relative entropy, one can refer to \cite{LL,GBM,M,MRW} by the technique of uniform in time log-Sobolev inequality of the limit equation. As to the uniform in time propagation of chaos in $\W_1$-distance, see \cite{DEGZ,GBMEJP,LWZ,LMW} for the method of (asymptotic) reflecting coupling or (asymptotic) refined basic coupling.

When $b_t(x,\mu)=\int_{\R^d}\tilde{b}(x,y)\mu(\d y)$, $\sigma_t(x,\mu)=\int_{\R^d}\tilde{\sigma}(x,y)\mu(\d y)$ for some Lipschitz continuous functions $\tilde{b},\tilde{\sigma}$, $X_0^{i,N}=X_0^i, 1\leq i\leq N$, $\E|X_0^1|^2<\infty$, \cite{SZ} derives propagation of chaos in strong sense, i.e. the estimate for $\E|X_t^{1,N}-X_t^1|^2$ by using synchronous coupling argument and hence the propagation of chaos in $\W_\theta$-distance for $\theta\in[1,2]$ is obtained due to the fact that for any $\theta\in[1,2]$ and $1\leq k\leq N$,
$$\W_{\theta}(\L_{(X_t^{i,N})_{1\leq i\leq k}},\L_{(X_t^{i})_{1\leq i\leq k}})\leq \W_{2}(\L_{(X_t^{i,N})_{1\leq i\leq k}},\L_{(X_t^{i})_{1\leq i\leq k}})\leq \sqrt{k\E|X_t^{1,N}-X_t^1|^2}.$$
However, the above inequality does not hold for $\theta\in(0,1)$ since $\W_\theta$ for $\theta\in(0,1)$ is not comparable with $\W_2$. Instead, it follows from Jensen's inequality that for $\theta\in(0,1)$ that
$$\W_{\theta}(\L_{(X_t^{i,N})_{1\leq i\leq k}},\L_{(X_t^{i})_{1\leq i\leq k}})\leq \W_{2}(\L_{(X_t^{i,N})_{1\leq i\leq k}},\L_{(X_t^{i})_{1\leq i\leq k}})^\theta\leq (k\E|X_t^{1,N}-X_t^1|^2)^{\frac{\theta}{2}}.$$
This inequality will reduce the convergence rate of propagation of chaos in $\W_\theta$ for $\theta\in(0,1)$. Consequently, it seems that the synchronous coupling argument is no longer appropriate to investigate the quantitative propagation of chaos in $\W_\theta$ for $\theta\in(0,1)$.

As far as we know, there is not any result on the propagation of chaos in $L^\eta(\eta\in(0,1))$-Wasserstein distance for mean field particle system with interacting diffusion coefficients.
 In this paper, we will fill this gap in both non-degenerate and kinetic cases. To this end, we will first derive the strong propagation of chaos and then adopt the dual formula \eqref{dfa} instead of \eqref{dew}. The main idea is to establish Duhamel's formula between \eqref{MVD} and \eqref{mfi}. To this end, the backward Kolmogorov equation as well as gradient estimate for the decoupled SDE plays a crucial role.


The remaining of the paper is organized as follows: In section 2,  we study the quantitative  propagation of chaos in $\W_\eta(\eta\in(0,1)$-distance for non-degenerate mean field interacting particle system. In Section 3, the second order mean field interacting particle system model is investigated.


\section{Non-degenerate case}
For any $F\in C_b^1((\R^d)^N)$,
$(x^1,x^2,\cdots,x^N)\in(\R^d)^N$, and $1\leq i\leq N$, let $\nabla_i F(x^1,x^2,\cdots,x^N)$ denote the gradient of $F$ with respect to the $i$-th component $x^i$. Denote $\nabla^2_i=\nabla_i\nabla_i$.

In this section, we assume that $b_t(x,\mu)=\int_{\R^d}b^{(1)}_t(x,y)\mu(\d y)$, $\sigma_t(x,\mu)=\int_{\R^d}\tilde{\sigma}_t(x,y)\mu(\d y)$ for some measurable $b^{(1)}:[0,T]\times \R^d\times\R^d\to\R^d$ and $\tilde{\sigma}:[0,T]\times \R^d\times\R^d\to\R^d\otimes\R^n$.
\begin{enumerate}
\item[{\bf(A1)}] There exist constants $K_\sigma>0$ and $\delta\geq 1$ such that
\begin{align*}
&\|\nabla\tilde{\sigma}_t(\cdot,y)(x)\|+\|\nabla^2\tilde{\sigma}_t(\cdot,y)(x)\|\leq K_\sigma,\ \ \delta^{-1}\leq\sigma\sigma^\ast\leq \delta,\\
&\|\tilde{\sigma}_t(x,y)-\tilde{\sigma}_t(x,\tilde{y})\|_{HS}\leq K_\sigma|y-\tilde{y}|,\ \ t\in[0,T], x,y,\tilde{y}\in\R^d.
\end{align*}
\item[{\bf(A2)}] There exists a constant $K_b>0$ such that
    \begin{align*}
    &\|\nabla b^{(1)}_t(\cdot,y)(x)\|+\|\nabla^2b^{(1)}_t(\cdot,y)(x)\|\leq K_b,\ \ |b^{(1)}_t(0,0)|\leq K_b,\\
    & |b^{(1)}_t(x,y)-b^{(1)}_t(x,\tilde{y})|\leq K_b|y-\tilde{y}|,\ \ t\in[0,T], x,y,\tilde{y}\in\R^d.
    \end{align*}
\end{enumerate}
Under {\bf(A1)}-{\bf(A2)}, \eqref{MVD} and \eqref{mfi} are well-posed. Throughout this section, let $\mu_t=\L_{X_t^i}$, which is independent of $i$ due to the weak uniqueness of \eqref{MVD}.
The following lemma is from \cite[Lemma 2.1]{HX23e}, which is useful in the proof of propagation of chaos. We list it here for readers' convenience.
\begin{lem}\label{CTY} Let $(V,\|\cdot\|_V)$ be a Banach space. $(Z_i)_{i\geq 1}$ are i.i.d. $V$-valued random variables with $\E\|Z_1\|^2_V<\infty$ and $h:V\times V\to \R$ is  measurable and of at most linear growth, i.e. there exists a constant $c>0$ such that
$$|h(v,\tilde{v})|\leq c(1+\|v\|_V+\|\tilde{v}\|_V),\ \ v,\tilde{v}\in V.$$
Then there exists a constant $\tilde{c}>0$ such that
\begin{align*}\E\left|\frac{1}{N}\sum_{m=1}^N h(Z_1,Z_m)-\int_{V} h(Z_1,y)\L_{Z_1}(\d y)\right|^2\leq \frac{\tilde{c}}{N}\E(1+\|Z_1\|_{V}^2).
\end{align*}
\end{lem}

\begin{thm}\label{POC10}
Assume {\bf(A1)}-{\bf(A2)}, $\L_{X_0^{1,N}}\in\scr P_{1}(\R^{d})$ and $\L_{X_0^{1}}\in \scr P_2(\R^{d})$. Then the following assertions hold.

(1) There exists a constant $C>0$ depending on $T$ such that
\begin{equation}\begin{split}\label{S1}
&\sum_{i=1}^N\E\sup_{t\in[0,T]}|X^{i,N}_t-X^i_t|\leq C\sum_{i=1}^{N}\E|X_0^{i,N}-X_0^{i}|+C\sqrt{N}(1+\{\E|X_0^1|^2\}^{\frac{1}{2}}).
\end{split}\end{equation}
(2) There exists a constant $c>0$ depending on $T$ such that for any $\eta\in(0,1]$, $\mu_0\in \scr P_2(\R^d)$ and exchangeable $\mu_0^N\in \scr P_1((\R^d)^N)$,
\begin{align}\label{CMY}\nonumber&\tilde{\W}_\eta((P_t^{[k],N})^\ast\mu_0^N,(P_t^\ast\mu_0)^{\otimes k})\\
&\leq \frac{k}{N}\min\left\{\left(\frac{c}{\eta}+c t^{\frac{-1+\eta}{2}}\right)\tilde{\W}_1(\mu_0^N,\mu_0^{\otimes N}), \quad \frac{c}{\eta}\tilde{\W}_1(\mu_0^N,\mu_0^{\otimes N})+c\tilde{\W}_\eta(\mu_0^N,\mu_0^{\otimes N})\right\}\\
\nonumber&+\frac{c}{\eta}\frac{k}{\sqrt{N}}(1+\{\mu_0(|\cdot|^2)\}^{\frac{1}{2}},\ \ t\in(0,T], 1\leq k\leq N.
\end{align}
\end{thm}

\begin{proof}
(1) Firstly, it is standard to derive from {\bf(A1)}-{\bf(A2)} that
\begin{align}\label{mmo} \E\left[1+\sup_{t\in[0,T]}|X^{1}_t|^2\right]<c_{1,T}\left(1+\E|X^1_0|^2\right)
\end{align}
and
\begin{align}\label{gro} \E\left[1+\sup_{t\in[0,T]}|X^{1,N}_t|\right]<c_{1,T}\left(1+\E|X^{1,N}_0|\right)
\end{align}
for some constant $c_{1,T}>0$.
 Combining \eqref{MVD} and \eqref{mfi}, we conclude
\begin{align}\label{GYG01}
\nonumber&\sum_{i=1}^N\E\sup_{t\in[0,s]}|X^{i,N}_t-X^{i}_t|\\
&\leq \sum_{i=1}^N\E|X^{i,N}_0-X^{i}_0|+\sum_{i=1}^N\E\int_{0}^s\left|\frac{1}{N}\sum_{m=1}^Nb_r^{(1)}(X_r^{i,N}, X_r^{m,N})-\int_{\R^d}b_r^{(1)}(X_r^{i}, y)\mu_r(\d y)\right|\d r\\
\nonumber&+\sum_{i=1}^N\E\sup_{t\in[0,s]}\left|\int_0^t\left(\frac{1}{N}\sum_{m=1}^N\tilde{\sigma}_r(X_r^{i,N}, X_r^{m,N})-\int_{\R^d}\tilde{\sigma}_r(X_r^{i}, y)\mu_r(\d y)\right)\d W_r^i\right|\\
\nonumber&=:\sum_{i=1}^N\E|X^{i,N}_0-X^{i}_0|+I_1+I_2,\ \ s\in[0,T].
\end{align}
Firstly, it follows from Lemma \ref{CTY}, \eqref{mmo} and {\bf (A2)} that there exists a constant $c_1>0$ depending on $T$ such that
\begin{align}\label{b-b}
\nonumber I_1&\leq \sum_{i=1}^N\E\int_{0}^s\left|\frac{1}{N}\sum_{m=1}^Nb_r^{(1)}(X_r^{i,N}, X_r^{m,N})-\frac{1}{N}\sum_{m=1}^Nb_r^{(1)}(X_r^{i}, X_r^{m})\right|\d r\\
&+\sum_{i=1}^N\E\int_{0}^s\left|\frac{1}{N}\sum_{m=1}^Nb_r^{(1)}(X_r^{i}, X_r^{m})-\int_{\R^d}b_r^{(1)}(X_r^{i}, y)\mu_r(\d y)\right|\d r\\
\nonumber&\leq c_1\int_0^s\sum_{i=1}^N\E|X_r^{i,N}-X_r^i|\d r+c_1\sqrt{N}\{1+(\E|X_0^{1}|^2)^{\frac{1}{2}}\},\ \ s\in[0,T].
\end{align}
Next, the BDG inequality, Lemma \ref{CTY}, \eqref{mmo} and {\bf (A1)} imply
\begin{align}\label{GYG02}
\nonumber&I_2\leq 3\sum_{i=1}^N\E\left(\int_0^s\left\|\frac{1}{N}\sum_{m=1}^N\tilde{\sigma}_r(X_r^{i,N}, X_r^{m,N})-\int_{\R^d}\tilde{\sigma}_r(X_r^{i}, y)\mu_r(\d y)\right\|_{HS}^2\d r\right)^{\frac{1}{2}}\\
\nonumber&\leq 3\sqrt{2}\sum_{i=1}^N\E\left(\int_0^s\left\|\frac{1}{N}\sum_{m=1}^N\tilde{\sigma}_r(X_r^{i,N}, X_r^{m,N})-\frac{1}{N}\sum_{m=1}^N\tilde{\sigma}_r(X_r^{i}, X_r^{m})\right\|_{HS}^2\d r\right)^{\frac{1}{2}}\\
&+3\sqrt{2}\sum_{i=1}^N\E\left(\int_0^s\left\|\frac{1}{N}\sum_{m=1}^N\tilde{\sigma}_r(X_r^{i}, X_r^{m})-\int_{\R^d}\tilde{\sigma}_r(X_r^{i}, y)\mu_r(\d y)\right\|_{HS}^2\d r\right)^{\frac{1}{2}}\\
\nonumber&\leq \frac{1}{4}\sum_{i=1}^N\E\sup_{r\in[0,s]}\left(\frac{1}{N}\sum_{i=1}^N|X_r^{m,N}-X_r^m|+|X_r^{i,N}-X_r^i|\right)\\ \nonumber&+c_2\int_0^s\sum_{i=1}^N\E|X_r^{i,N}-X_r^i|\d r+c_2\sqrt{N}\{1+(\E|X_0^{1}|^2)^{\frac{1}{2}}\},\ \ s\in[0,T]
\end{align}
for some constant $c_2>0$ depending on $T$.
Combining \eqref{GYG01}-\eqref{GYG02}, we find a constant $c_{3}>0$ depending on $T$ such that
\begin{align}\label{GYG}
\sum_{i=1}^N\E\sup_{t\in[0,s]}|X^{i,N}_t-X^{i}_t|&\leq 2\sum_{i=1}^N\E|X^{i,N}_0-X^{i}_0|+c_{3}\int_{0}^s\sum_{i=1}^N\E\sup_{t\in[0,r]}|X^{i,N}_t-X^{i}_t|\d r\\
\nonumber&+c_3\sqrt{N}\{1+(\E|X_0^{1}|^2)^{\frac{1}{2}}\}
,\ \ s\in[0,T].
\end{align}
Then \eqref{S1} follows from \eqref{GYG}, \eqref{mmo}, \eqref{gro} and Gr\"{o}nwall's inequality.

(2)
For $0\leq s\leq t\leq T$, consider the decoupled SDE $$\d X_{s,t}^{\mu,x}=b_t(X_{s,t}^{\mu,x},\mu_t)\d t+\sigma_t(X_{s,t}^{\mu,x},\mu_t)\d W_t^1,\ \ X_{s,s}^{\mu,x}=x\in\R^d.$$
Let
$$P_{s,t}^\mu f(x):=\E f(X_{s,t}^{\mu,x}), \ \ f\in \scr B_b(\R^{d}),x\in\R^d,$$
and for any $x=(x^1,x^2,\cdots,x^N)\in (\R^{d})^N, F\in \scr B_b((\R^{d})^N)$,
\begin{align*}(P_{s,t}^\mu)^{\otimes N} F(x)&:=\int_{(\R^{d})^N}F(y^1,y^2,\cdots,y^N)\prod_{i=1}^N\L_{X_{s,t}^{\mu,x^i}}(\d y^i).
\end{align*}
For simplicity, we write $P_{t}^\mu =P_{0,t}^\mu $. Take $(X_0^{i,N})_{1\leq i\leq N}$ and $(X_0^{i})_{1\leq i\leq N}$ satisfying $\L_{(X_0^{i,N})_{1\leq i\leq N}}=\mu_0^N$ and $\L_{(X_0^{i})_{1\leq i\leq N}}=\mu_0^{\otimes N}$. Consider
\begin{equation}\label{eq4}
\d \bar{X}_t^{i}=b_t(\bar{X}_t^{i},
\mu_t)\d t+\si_t(\bar{X}_t^{i},\mu_t)\d
W_t^i,\ \ \bar{X}_0^i=X_0^{i,N}, 1\leq i\leq N.
\end{equation}
In the following, we divide the proof into two steps.

{\bf Step 1. Estimate $\tilde{\W}_{\eta}((P_t^{[N],N})^\ast\mu_0^N,\L_{(\bar{X}_t^{i})_{1\leq i\leq N}})$.}

Firstly, we claim that there exists a constant $C_0>0$ depending on $T$ such that for any $\eta\in(0,1]$,
\begin{align}\label{PTf}
&\tilde{\W}_{\eta}((P_t^{[N],N})^\ast\mu_0^N,\L_{(\bar{X}_t^{i})_{1\leq i\leq N}})\leq \frac{C_0}{\eta}\left\{\tilde{\W}_1(\mu_0^N,\mu_0^{\otimes N})+\sqrt{N}(1+\{\E|X_0^1|^2\}^{\frac{1}{2}})\right\} .
    \end{align}
For any $t\in[0,T], x=(x^1,x^2,\cdots,x^N)\in (\R^{d})^N$, let $$B_t(x)=\left(\frac{1}{N}\sum_{m=1}^Nb_t^{(1)}(x^i,x^m)\right)_{1\leq i\leq N},\ \ \Sigma_t(x)=\mathrm{diag}\left(\left(\frac{1}{N}\sum_{m=1}^N\tilde{\sigma}_t(x^i,x^m)\right)_{1\leq i\leq N}\right),$$
and
$$B_t^\mu(x)=\left(\int_{\R^d}b_t^{(1)}(x^i,y)\mu_t(\d y)\right)_{1\leq i\leq N},\ \ \Sigma_t^\mu(x)=\mathrm{diag}\left(\left(\int_{\R^d}\tilde{\sigma}_t(x^i,y)\mu_t(\d y)\right)_{1\leq i\leq N}\right).$$
We denote
$$\scr L_t^{\mu,N}=\<B_t^\mu,\nabla\>+\frac{1}{2}\mathrm{tr}(\Sigma_t^\mu(\Sigma_t^\mu)^\ast\nabla^2),\ \ \scr L_t^{N}=\<B_t,\nabla\>+\frac{1}{2}\mathrm{tr}(\Sigma_t(\Sigma_t)^\ast\nabla^2), \ \ t\in[0,T].$$
By {\bf (A1)}-{\bf (A2)}, it is well-known that $P_{s,t}^\mu C_b^2(\R^d)\subset C_b^2(\R^d)$ and there exists a constant $c>0$ depending on $T$ such that
\begin{align}\label{gre}
|\nabla^{j}P_{s,t}^\mu f|\leq c(t-s)^{-\frac{j}{2}}\{P_{s,t}^\mu |f|^2\}^{\frac{1}{2}},\ \ 0\leq s<t\leq T, f\in C_b^2(\R^d),j=1,2,
\end{align}
see for instance \cite[Proof of Lemma 2.1]{Wang16} and \cite{Nbook}. Hence, it follows from It\^{o}'s formula that the backward Kolmogorov equations holds:
\begin{align}\label{bke}
\frac{\d \{(P_{s,t}^\mu)^{\otimes N} F\}}{\d s}=-\scr L^{\mu,N}_s\{(P_{s,t}^\mu)^{\otimes N} F\},\ \ F\in C_b^2((\R^d)^N),\ \ 0\leq s\leq t\leq T.
\end{align}
Fix $t\in(0,T]$. Then \eqref{bke} together with It\^{o}'s formula implies that for any $F\in C_b^2((\R^d)^N)$,
\begin{align}\label{itf}
\nonumber&\d \{(P_{s,t}^\mu)^{\otimes N} F\}(X_s^{1,N},X_s^{2,N},\cdots,X_s^{N,N})\\
&=(\scr L_s^{N}-\scr L^{\mu,N}_s)\{(P_{s,t}^\mu)^{\otimes N} F\}(X_s^{1,N},X_s^{2,N},\cdots,X_s^{N,N})\d s+\d M_s,\ \ s\in[0,t]
\end{align}
holds for some martingale $\{M_s\}_{s\in[0,t]}$.
So, \eqref{itf} combined with the fact
$$\E[\{(P_{t}^\mu)^{\otimes N} F\}(X_0^{1,N},X_0^{2,N},\cdots,X_0^{N,N})]=\E F(\bar{X}_t^{1},\bar{X}_t^{2},\cdots,\bar{X}_t^{N}), \ \ F\in C_b^2((\R^d)^N)$$
implies
\begin{align}\label{DUH}
\nonumber&\E F(X_t^{1,N},X_t^{2,N},\cdots,X_t^{N,N})-\E F(\bar{X}_t^{1},\bar{X}_t^{2},\cdots,\bar{X}_t^{N})\\
&=\E\int_0^t(\scr L_s^{N}-\scr L^{\mu,N}_s)\{(P_{s,t}^\mu)^{\otimes N} F\}(X_s^{1,N},X_s^{2,N},\cdots,X_s^{N,N})\d s \\
\nonumber&=R_{0,t}^{b}F+R_{0,t}^{\sigma}F, \ \ F\in C_b^2((\R^d)^N)
\end{align}
    with
\begin{align*}
    \nonumber&R_{0,t}^{b}F:=\E\int_0^t \left\{\<B_s-B_s^\mu,\nabla(P^\mu_{s,t})^{\otimes N} F\>\right\}(X_s^{1,N},X_s^{2,N},\cdots,X_s^{N,N})\d s\\
  &=\int_0^t \E\sum_{i=1}^N\Bigg\<\frac{1}{N}\sum_{m=1}^Nb_s^{(1)}(X_{s}^{i,N},X_{s}^{m,N})-\int_{\R^d} b_s^{(1)}(X_{s}^{i,N},y)\mu_s(\d y), \\
\nonumber &\qquad\qquad\qquad\qquad\quad[\nabla_{i}(P^\mu_{s,t})^{\otimes N}F](X_{s}^{1,N},X_{s}^{2,N},\cdots,X_{s}^{N,N})\Bigg\>\d s,
   \end{align*}
   and
 \begin{align*}
    \nonumber&R_{0,t}^{\sigma}F=\frac{1}{2}\int_0^t \E\mathrm{tr}\left\{[\Sigma_s\Sigma^\ast_s-\Sigma_s^\mu(\Sigma^\mu_s)^\ast]\nabla^2 (P^{\mu}_{s,t})^{\otimes N} F\right\}(X_s^{1,N},X_s^{2,N},\cdots,X_s^{N,N})\d s\\
    &=\frac{1}{2}\int_0^t \E\sum_{i=1}^N\mathrm{tr}\Bigg \{\bigg[\bigg(\frac{1}{N}\sum_{m=1}^N\tilde{\sigma}_s(X_{s}^{i,N},X_{s}^{m,N}) \frac{1}{N}\sum_{m=1}^N\tilde{\sigma}_s^\ast(X_{s}^{i,N},X_{s}^{m,N})\bigg)\\
   \nonumber &-\int_{\R^d}\tilde{\sigma}_s(X_{s}^{i,N},y)\mu_s(\d y)\int_{\R^d}\tilde{\sigma}_s^\ast(X_{s}^{i,N},y)\mu_s(\d y)\bigg][\nabla^2_{i}(P^\mu_{s,t})^{\otimes N}F](X_{s}^{1,N},X_{s}^{2,N},\cdots,X_{s}^{N,N})\Bigg\}\d s.
\end{align*}
Consider
$$\ff{\d}{\d t} \theta_{s,t}(x)=b_t(\theta_{s,t}(x),\mu_t),\ \ \ t\in [s,T], \theta_{s,s}(x)=x.$$
By  It\^o's formula, {\bf (A1)}-{\bf (A2)}, we find a constant $\tilde{c}_1>0$ and a martingale $\tilde{M}_t$  such that
\beg{align*} &\d |X_{s,t}^{\mu,x}-\theta_{s,t}(x)|^2\\
 &= \Big\{2 \big\<X_{s,t}^{\mu,x}-\theta_{s,t}(x),  b_t(X_{s,t}^{\mu,x},\mu_t) -b_t(\theta_{s,t}(x),\mu_t)\big\> + \| \si_t (X_{s,t}^{\mu,x},\mu_t)\|_{HS}^2 \Big\}\d t + \d \tilde{M}_t\\
&\le \tilde{c}_1 \big\{|X_{s,t}^{\mu,x}-\theta_{s,t}(x)|^2 +1\Big\}\d t+\d \tilde{M}_t,\ \ \ t\in [s,T], |X_{s,s}^{\mu,x}-\theta_{s,s}(x)|=0.\end{align*}
So, we have
$$\E\big[|X_{s,t}^{\mu,x}-\theta_{s,t}(x)|^2\big]\le \tilde{c}_1 \e^{\tilde{c}_1 T}(t-s),\ \ \ 0\le s\le t\le T.$$
Combining this with \eqref{gre}, we conclude that for any $f\in C_b^2(\R^d)$ with $[f]_\eta<1$,
\begin{align}\label{gra}
\nonumber|\nabla^{j}P_{s,t}^\mu f|(x)&=|\nabla^{j}P_{s,t}^\mu (f-f(\theta_{s,t}(x)))|(x)\leq c(t-s)^{-\frac{j}{2}}\{P_{s,t}^\mu |f-f(\theta_{s,t}(x))|^2(x)\}^{\frac{1}{2}}\\
&\leq c(t-s)^{-\frac{j}{2}}\{\E|X_{s,t}^{\mu,x}-\theta_{s,t}(x)|^2\}^{\frac{\eta}{2}}\\
\nonumber&\leq c_0(t-s)^{\frac{-j+\eta}{2}},\ \ 0\leq s<t\leq T,j=1,2
\end{align}
holds for some constant $c_0>0$ depending on $T$.

Let $g\in C_b^2((\R^d)^N)$ with $[g]_{1,\eta}\leq 1$ for $[\cdot]_{1,\eta}$ defined in \eqref{1et}. For any $1\leq i\leq N$, $(y^1,y^2,\cdots,y^{i-1},y^{i+1},\cdots, y^N)\in (\R^{d})^{ (N-1)}$, define
\begin{align*}
&[\scr M_{s,t}^{i,\mu,(y^1,y^2,\cdots,y^{i-1},y^{i+1},\cdots, y^N)}(g)](z)\\
&=\int_{(\R^{d})^{ (N-1)}}g(z^1,z^2,\cdots, z^{i-1},z,z^{i+1},\cdots,z^N)\prod_{m=1,m\neq i}^N\L_{X_{s,t}^{\mu,y^m}}(\d z^m),\ \  z\in\R^d.
\end{align*}
It is not difficult to see that for any $1\leq i\leq N$, $(y^1,y^2,\cdots,y^{i-1},y^{i+1},\cdots, y^N)\in (\R^{d})^{ (N-1)}$,
\begin{align*}&\left|[\scr M_{s,t}^{i,\mu,(y^1,y^2,\cdots,y^{i-1},y^{i+1},\cdots, y^N)}(g)](z)-[\scr M_{s,t}^{i,\mu,(y^1,y^2,\cdots,y^{i-1},y^{i+1},\cdots, y^N)}(g)](\tilde{z})\right|\\
&\qquad\quad\leq |z-\tilde{z}|^\eta,\ \ z,\tilde{z}\in\R^d.
\end{align*}
This together with \eqref{gra} implies that
\begin{align}\label{kb2}|\nabla^j_{i}(P^\mu_{s,t})^{\otimes N}g|(y^1,y^2,\cdots,y^N)&=\left|\nabla^j_{i}\{P_{s,t}^\mu[\scr M_{s,t}^{i,\mu,(y^1,y^2,\cdots,y^{i-1},y^{i+1},\cdots, y^N)}(g)]\}(y^i)\right|\\
\nonumber&\leq c_0(t-s)^{\frac{-j+\eta}{2}},\ \ j=1,2, 1\leq i\leq N.
\end{align}
Substituting \eqref{kb2} into \eqref{DUH} by replacing $F$ with $g$ and combining the fact $\|\sigma\|\leq \sqrt{\delta}$ due to {\bf(A1)}, we arrive at
\begin{align}\label{DUE}
\nonumber&|\E g((X_t^{i,N})_{1\leq i\leq N})-\E g((\bar{X}_t^i)_{1\leq i\leq N})|\\
\nonumber&\leq c_0\int_0^t \sum_{i=1}^N\E\left|\frac{1}{N}\sum_{m=1}^Nb_r^{(1)}(X_r^{i,N},X_r^{m,N})-\int_{\R^d}b_r^{(1)}(X_r^{i,N},y)\mu_r(\d y)\right|(t-r)^{\frac{-1+\eta}{2}}\d r\\
&+c_0\int_0^t \sum_{i=1}^N\E\Bigg\|\frac{1}{N^2}\sum_{m=1}^N \sum_{l=1}^N\tilde{\sigma}_r(X_r^{i,N},X_r^{m,N})\tilde{\sigma}_r^\ast(X_r^{i,N},X_r^{l,N}) \\
\nonumber&\qquad\qquad-\int_{\R^d}\tilde{\sigma}_r(X_r^{i,N},y)\mu_r(\d y)\int_{\R^d}\tilde{\sigma}_r^\ast(X_r^{i,N},y)\mu_r(\d y)
    \Bigg\|_{HS}\times(t-r)^{\frac{-2+\eta}{2}}\d r\\
\nonumber&\leq c_0\int_0^t \sum_{i=1}^N\E\left|\frac{1}{N}\sum_{m=1}^Nb_r^{(1)}(X_r^{i,N},X_r^{m,N})-\int_{\R^d}b_r^{(1)}(X_r^{i,N},y)\mu_r(\d y)\right|(t-r)^{\frac{-1+\eta}{2}}\d r\\
\nonumber&+2c_0\sqrt{\delta}\int_0^t \sum_{i=1}^N\E\Bigg\|\frac{1}{N}\sum_{m=1}^N\tilde{\sigma}_r(X_r^{i,N},X_r^{m,N}) -\int_{\R^d}\tilde{\sigma}_r(X_r^{i,N},y)\mu_r(\d y)
    \Bigg\|_{HS}(t-r)^{\frac{-2+\eta}{2}}\d r.
\end{align}
By {\bf (A2)}, \eqref{mmo}, \eqref{S1} and Lemma \ref{CTY}, we can find a constant $C>0$ depending on $T$ such that
\begin{align}\label{AA1}
\nonumber&\sum_{i=1}^N\E\left|\frac{1}{N}\sum_{m=1}^N  b_r^{(1)}(X_r^{i,N},X_r^{m,N})-\int_{\R^d} b_r^{(1)}(X_r^{i,N},y)\mu_r(\d y)\right|\\
\nonumber&\leq\sum_{i=1}^N \E\bigg|\frac{1}{N}\sum_{m=1}^N b_r^{(1)}(X_r^{i,N},X_r^{m,N})-\int_{\R^d} b_r^{(1)}(X_r^{i,N},y)\mu_r(\d y)\\
&\qquad\qquad\quad-\bigg(\frac{1}{N}\sum_{m=1}^N  b_r^{(1)}(X_r^{i},X_r^{m})-\int_{\R^d} b_r^{(1)}(X_r^{i},y)\mu_r(\d y)\bigg)\bigg|\\
\nonumber&+\sum_{i=1}^N\E\left|\frac{1}{N}\sum_{m=1}^N  b_r^{(1)}(X_r^{i},X_r^{m})-\int_{\R^d} b_r^{(1)}(X_r^{i},y)\mu_r(\d y)\right|\\
\nonumber&\leq 3K_b\sum_{i=1}^N\E|X^{i,N}_r-X^{i}_r|+c\sqrt{N}\{1+(\E|X_0^{1}|^2)^{\frac{1}{2}}\}\\
\nonumber&\leq C\sum_{i=1}^{N}\E|X_0^{i,N}-X_0^{i}|+C\sqrt{N}(1+\{\E|X_0^1|^2\}^{\frac{1}{2}}).
\end{align}
Similarly, {\bf(A1)}, \eqref{mmo}, \eqref{S1} and Lemma \ref{CTY} imply
\begin{align}\label{AA2}
\nonumber&\sum_{i=1}^N\E\left\|\frac{1}{N}\sum_{m=1}^N \tilde{\sigma}_r(X_r^{i,N},X_r^{m,N})-\int_{\R^d}\tilde{\sigma}_r(X_r^{i,N},y)\mu_r(\d y)\right\|_{HS}\\
&\leq \tilde{C}\sum_{i=1}^{N}\E|X_0^{i,N}-X_0^{i}|+\tilde{C}\sqrt{N}(1+\{\E|X_0^1|^2\}^{\frac{1}{2}})
\end{align}
foe some constant $\tilde{C}>0$ depending on $T$.
Substituting \eqref{AA1}-\eqref{AA2} into \eqref{DUE}, we may find a constant $C_1>0$ such that for any $g\in C_b^2((\R^d)^N)$ with $[g]_{1,\eta}\leq 1$,
\begin{align*}
&|\E g((X_t^{i,N})_{1\leq i\leq N})-\E g((\bar{X}_t^i)_{1\leq i\leq N})|\\
\nonumber&\leq \left\{C_1\sum_{i=1}^{N}\E|X_0^{i,N}-X_0^{i}|+C_1\sqrt{N}(1+\{\E|X_0^1|^2\}^{\frac{1}{2}})\right\}\\
&\quad\quad\times \left\{\int_0^t (t-r)^{\frac{-1+\eta}{2}}\d r+\int_0^t (t-r)^{\frac{-2+\eta}{2}}\d r\right\}\\
&\leq \left\{2C_1\sum_{i=1}^{N}\E|X_0^{i,N}-X_0^{i}|+2C_1\sqrt{N}(1+\{\E|X_0^1|^2\}^{\frac{1}{2}}) \right\} \left\{t^{\frac{1+\eta}{2}}+\frac{t^{\frac{\eta}{2}}}{\eta}\right\}.
\end{align*}
By taking infinimum with respect to $(X_0^{i,N},X_0^i)_{1\leq i\leq N}$ satisfying $\L_{(X_0^{i,N})_{1\leq i\leq N}}=\mu_0^N$ and $\L_{(X_0^{i})_{1\leq i\leq N}}=\mu_0^{\otimes N}$, we derive \eqref{PTf} from \eqref{dfa1}.

{\bf Step 2. Estimate $\tilde{\W}_{\eta}(\L_{(\bar{X}_t^{i})_{1\leq i\leq N}},(P_t^\ast\mu_0)^{\otimes N})$.}

Define $\P^{0}:= \P(\ \cdot\ |\F_0),\ \ \E^{0}:= \E(\ \cdot\ | \F_0)$ and let $\L_{\xi|\P^0}$ denote the conditional distribution of a random variable $\xi$ with respect to $\F_0$.
By \eqref{gra} and \eqref{eq4}, we derive
\begin{align}\label{tri}\W_\eta(\L_{\bar{X}_t^i|\P^0},\L_{X_t^i|\P^0})\leq c_0 t^{\frac{-1+\eta}{2}}|X_0^{i,N}-X_0^{i}|,\ \ 1\leq i\leq N.
\end{align}
On the other hand, {\bf(A1)}-{\bf (A2)} imply
$$\E(|X_t^i-\bar{X}_t^i|^2|\F_0)\leq \tilde{c}|X_0^i-X_0^{i,N}|^2,\ \ t\in[0,T], 1\leq i\leq N$$
for some constant $\tilde{c}$ depending on $T$. Jensen's inequality gives
\begin{align}\label{efa}\W_\eta(\L_{\bar{X}_t^i|\P^0},\L_{X_t^i|\P^0})\leq \E(|X_t^i-\bar{X}_t^i|^\eta|\F_0)\leq \tilde{c}^{\frac{\eta}{2}}|X_0^i-X_0^{i,N}|^\eta,\ \ t\in[0,T], 1\leq i\leq N.
\end{align}
Since both $(\bar{X}_t^i)_{1\leq i\leq N}$ and $(X_t^i)_{1\leq i\leq N}$ are independent under $\P^0$, we have
\begin{align*}\tilde{\W}_\eta(\L_{(\bar{X}_t^i)_{1\leq i\leq N}|\P^0},\L_{(X_t^i)_{1\leq i\leq N}|\P^0}) \leq \sum_{i=1}^N\W_\eta(\L_{\bar{X}_t^i|\P^0},\L_{X_t^i|\P^0}).
\end{align*}
This together with \eqref{tri} and \eqref{efa} implies
\begin{align*}&\tilde{\W}_\eta(\L_{(\bar{X}_t^i)_{1\leq i\leq N}},(P_t^\ast\mu_0)^{\otimes N})\\
&\leq \E\{\tilde{\W}_\eta(\L_{(\bar{X}_t^i)_{1\leq i\leq N}|\P^0},\L_{(X_t^i)_{1\leq i\leq N}|\P^0})\} \\
&\leq \min\left\{c_0 t^{\frac{-1+\eta}{2}}\sum_{i=1}^N\E|X_0^{i,N}-X_0^{i}|, \quad \tilde{c}^{\frac{\eta}{2}}\sum_{i=1}^N\E|X_0^{i,N}-X_0^{i}|^\eta\right\}.
\end{align*}
By taking infinimum with respect to $(X_0^{i,N},X_0^i)_{1\leq i\leq N}$ satisfying $\L_{(X_0^{i,N})_{1\leq i\leq N}}=\mu_0^N$ and $\L_{(X_0^{i})_{1\leq i\leq N}}=\mu_0^{\otimes N}$, we get
\begin{align}\label{tri13}\tilde{\W}_\eta(\L_{(\bar{X}_t^i)_{1\leq i\leq N}},(P_t^\ast\mu_0)^{\otimes N})\leq \min\left\{c_0 t^{\frac{-1+\eta}{2}}\tilde{\W}_1(\mu_0^N,\mu_0^{\otimes N}), \quad \tilde{c}^{\frac{\eta}{2}}\tilde{\W}_\eta(\mu_0^N,\mu_0^{\otimes N})\right\}.
\end{align}
Finally, by \eqref{PTf}, \eqref{tri13} and the triangle inequality, we obtain
\begin{align*}
&\tilde{\W}_{\eta}((P_t^{[N],N})^\ast\mu_0^N,(P_t^\ast\mu_0)^{\otimes N})\\
&\leq \frac{C_0}{\eta}\left\{\tilde{\W}_1(\mu_0^N,\mu_0^{\otimes N})+\sqrt{N}(1+\{\E|X_0^1|^2\}^{\frac{1}{2}})\right\}\\
&+\min\left\{c_0 t^{\frac{-1+\eta}{2}}\tilde{\W}_1(\mu_0^N,\mu_0^{\otimes N}), \quad \tilde{c}^{\frac{\eta}{2}}\tilde{\W}_\eta(\mu_0^N,\mu_0^{\otimes N})\right\}\\
&\leq \min\left\{\left(\frac{C_0}{\eta}+c_0 t^{\frac{-1+\eta}{2}}\right)\tilde{\W}_1(\mu_0^N,\mu_0^{\otimes N}), \quad \frac{C_0}{\eta}\tilde{\W}_1(\mu_0^N,\mu_0^{\otimes N})+\tilde{c}^{\frac{\eta}{2}}\tilde{\W}_\eta(\mu_0^N,\mu_0^{\otimes N})\right\}\\
&+\frac{C_0}{\eta}\sqrt{N}(1+\{\E|X_0^1|^2\}^{\frac{1}{2}}.
\end{align*}
This yields \eqref{CMY} in view of
\begin{align*}\tilde{\W}_\eta((P_t^{[k],N})^\ast\mu_0^N,(P_t^\ast\mu_0)^{\otimes k})&\leq \frac{k}{N}\tilde{\W}_\eta((P_t^{[N],N})^\ast\mu_0^N,(P_t^\ast\mu_0)^{\otimes N}).
\end{align*}
Hence, the proof is completed.
\end{proof}
\section{Second order system}
In this section, we consider the second order mean field interacting system, i.e. the kinetic case. We first introduce some notations. For any $F\in C_b^1(\R^{2d})$, $x,y\in\R^{d}$, let $\nabla_x F(x,y)$ and $\nabla_y F(x,y)$ denote the gradient with respect to $x$ and $y$ respectively. We denote $\nabla^2_y=\nabla_y\nabla_y$ and $\nabla^1_y=\nabla_y$. For any $k\geq 1$, $F\in C_b^1((\R^{2d})^k)$, $1\leq i\leq k$, $z=(z^1,z^2,\cdots,z^k)$ with $z^i=(x^i,y^i)$ and $x^i,y^i\in\R^d$, let $(\nabla_{i,x} F(z^1,z^2,\cdots,z^k),\nabla_{i,y} F(z^1,z^2,\cdots,z^k))$ denote the gradient with respect to $(x^i,y^i)$. Denote $\nabla^2_{i,y}=\nabla_{i,y}\nabla_{i,y}$.

Let $T>0$, $\{(W_t^i)_{t\in[0,T]}\}_{i\geq 1}$ be independent $n$-dimensional Brownian motions on some complete filtration probability space $(\Omega, \scr F, (\scr F_t)_{t\in[0,T]},\P)$ and $(X_0^i,Y_0^i)_{i\geq 1}$ be i.i.d. $\F_0$-measurable and $\R^{2d}$-valued random variables. $b:[0,T]\times \R^{2d}\times\scr P(\R^{2d})\to\R^d$, $\sigma:[0,T]\times \R^{2d}\times\scr P(\R^{2d})\to\R^d\otimes\R^{n}$ are measurable and  bounded on bounded sets. Consider
\begin{equation*}
\begin{cases}
\d X^{i}_t=Y^{i}_t\d t, \\
\d Y_t^i=b_t(X_t^i,Y_t^i, \L_{(X_t^i,Y_t^i)})\d t+  \sigma_t(X_t^i, Y_t^i, \L_{(X_t^i,Y_t^i)})\d W^i_t,\ \ 1\leq i\leq N,
\end{cases}
\end{equation*}
and the mean field interacting particle system
\begin{equation*}
\begin{cases}
\d X^{i,N}_t=Y^{i,N}_t\d t, \\
\d Y_t^{i,N}=b_t(X_t^{i,N},Y_t^{i,N}, \hat\mu_t^N)\d t+  \sigma_t(X_t^{i,N}, Y_t^{i,N}, \hat\mu_t^N)\d W^i_t,\ \ 1\leq i\leq N,
\end{cases}
\end{equation*}
where the distribution of $(X_0^{i,N},Y_0^{i,N})_{1\leq i\leq N}$ is exchangeable and
\begin{align*}\hat\mu_t^N=\frac{1}{N}\sum_{i=1}^N\delta_{(X_t^{i,N},Y_t^{i,N})},\ \ 1\leq i\leq N.
\end{align*}
As in Section 2, we assume that $b_t(x,\mu)=\int_{\R^{2d}}b^{(1)}_t(x,y)\mu(\d y)$, $\sigma_t(x,\mu)=\int_{\R^{2d}}\tilde{\sigma}_t(x,y)\mu(\d y)$ for some measurable $b^{(1)}:[0,T]\times \R^{2d}\times\R^{2d}\to\R^d$ and $\tilde{\sigma}:[0,T]\times \R^{2d}\times\R^{2d}\to\R^d\otimes\R^n$.
The main result is the following theorem.
\begin{thm}\label{POC12}
Assume {\bf(A1)}-{\bf(A2)} with $\R^{2d}$ in place of $\R^d$ and $\L_{(X_0^{1,N},Y_0^{1,N})}\in\scr P_{1}(\R^{2d})$, $\L_{(X_0^{1},Y_0^{1})}\in \scr P_2(\R^{2d})$.
Then the following assertions hold.

(1) There exists a constant $C>0$ depending only on $T$ such that
\begin{equation}\begin{split}\label{S12}
\sum_{i=1}^N\E\sup_{t\in[0,T]}|(X^{i,N}_t,Y_t^{i,N})-(X^i_t,Y_t^i)|&\leq C\sum_{i=1}^{N}\E|(X^{i,N}_0,Y_0^{i,N})-(X^i_0,Y_0^i)|\\
&+C\sqrt{N}(1+\{\E|(X_0^1,Y_0^1)|^2\}^{\frac{1}{2}}).
\end{split}\end{equation}
(2) There exists a constant $c>0$ depending only on $T$ such that for any $\eta\in(0,1]$,
\begin{align}\label{CMY1}\nonumber&\tilde{\W}_\eta(\L_{(X_t^{i,N},Y_t^{i,N})_{1\leq i\leq k}},\L_{(X_t^i,Y_t^i)_{1\leq i\leq k}})\\
&\leq \frac{k}{N}\min\left\{\left(\frac{c}{\eta}+c t^{\frac{-3+\eta}{2}}\right)\tilde{\W}_1,\quad(\frac{c}{\eta} \tilde{\W}_1+c\tilde{\W}_\eta)\right\}(\L_{(X_0^{i,N},Y_0^{i,N})_{1\leq i\leq N}},\L_{(X_0^i,Y_0^i)_{1\leq i\leq N}})\\
\nonumber&+\frac{c}{\eta}\frac{k}{\sqrt{N}}(1+\{\E|(X_0^1,Y_0^1)|^2\}^{\frac{1}{2}}), \ \ t\in(0,T], 1\leq k\leq N.
\end{align}
\end{thm}
\begin{proof} (1) The proof is completely the same as that of Theorem \ref{POC10}(1).

(2) The proof is more or less the same as that of Theorem \ref{POC10}(2). Let $\mu_t=\L_{(X_t^1,Y_t^1)}$. For $0\leq s\leq t\leq T$, letting $(X_{s,s}^{\mu,z},Y_{s,s}^{\mu,z})=z\in\R^{2d}$, consider the decoupled SDE $$\begin{cases}
\d X^{\mu,z}_{s,t}=Y_{s,t}^{\mu,z}\d t, \\
\d Y_{s,t}^{\mu,z}=b_t(X^{\mu,z}_{s,t},Y_{s,t}^{\mu,z}, \mu_t)\d t+  \sigma_t(X^{\mu,z}_{s,t}, Y^{\mu,z}_{s,t}, \mu_t)\d W_t^1,\ \ s\leq t\leq T.
\end{cases}$$
Let
$$P_{s,t}^\mu f(z):=\E f(X_{s,t}^{\mu,z},Y_{s,t}^{\mu,z}), \ \ f\in \scr B_b(\R^{2d}),z\in\R^{2d},$$
and for any $z=(z^1,z^2,\cdots,z^N)\in (\R^{2d})^N, F\in \scr B_b((\R^{2d})^N)$,
$$(P_{s,t}^\mu)^{\otimes N} F(z):=\int_{(\R^{2d})^N}F(w^1,w^2,\cdots,w^N)\prod_{i=1}^N\L_{(X_{s,t}^{\mu,z^i},Y_{s,t}^{\mu,z^i})}(\d w^i). $$
For simplicity, we write $P_{t}^\mu =P_{0,t}^\mu $.
letting $(\bar{X}^i_{0},\bar{Y}_{0}^{i})=(X_0^{i,N},Y_0^{i,N}), 1\leq i\leq N$, consider
$$\begin{cases}
\d \bar{X}^{i}_{t}=\bar{Y}_{t}^{i}\d t, \\
\d \bar{Y}^{i}_{t}=b_t(\bar{X}^{i}_{t},\bar{Y}^{i}_{t}, \mu_t)\d t+  \sigma_t(\bar{X}^{i}_{t}, \bar{Y}^{i}_{t}, \mu_t)\d W_t^i,\ \ 1\leq i\leq N.
\end{cases}$$
Similar to \eqref{gre} and \eqref{bke}, by {\bf (A1)}-{\bf (A2)}, it is well-known that $P_{s,t}^\mu C_b^2(\R^{2d})\subset C_b^2(\R^{2d})$ with
\begin{align}\label{gre1}
|\nabla_y^{j}P_{s,t}^\mu f|\leq c(t-s)^{-\frac{j}{2}}\{P_{s,t}^\mu |f|^2\}^{\frac{1}{2}},\ \ 0\leq s<t\leq T, f\in C_b^2(\R^{2d}),j=1,2,
\end{align}
\begin{align}\label{gre2}
|\nabla P_{s,t}^\mu f|\leq c(t-s)^{-\frac{3}{2}}\{P_{s,t}^\mu |f|^2\}^{\frac{1}{2}},\ \ 0\leq s<t\leq T, f\in C_b^1(\R^{2d}),
\end{align}
and for any $ F\in C_b^2((\R^{2d})^N)$, the backward Kolmogorov equation
\begin{align*}
&\frac{\d \{(P_{s,t}^\mu)^{\otimes N} F\}(z)}{\d s}\\
&=-\sum_{i=1}^N\{\<y^i,\nabla_{i,x}\{(P_{s,t}^\mu)^{\otimes N} F\}(z)\>+\<b_s(x^i,y^i,\mu_s),\nabla_{i,y}\{(P_{s,t}^\mu)^{\otimes N} F\}(z)\>\}\\
&-\frac{1}{2}\sum_{i=1}^N \mathrm{tr}((\sigma_s\sigma_s^\ast)(x^i,y^i,\mu_s)\nabla^2_{i,y}\{(P_{s,t}^\mu)^{\otimes N} F\}(z)),\ \ 0\leq s\leq t\leq T
\end{align*}
holds for $z=(z^1,z^2,\cdots, z^N)\in(\R^{2d})^N,z^i=(x^i,y^i),1\leq i\leq N$.
Let
\begin{align*}
B^i_s&=\frac{1}{N}\sum_{m=1}^Nb_s^{(1)}((X_{s}^{i,N},Y_{s}^{i,N}),(X_{s}^{m,N},Y_{s}^{m,N}))-\int_{\R^{2d}} b_s^{(1)}((X_{s}^{i,N},Y_{s}^{i,N}),y)\mu_s(\d y),
\end{align*}
and
\begin{align*}\Sigma^i_s
&=\frac{1}{N}\sum_{m=1}^N\tilde{\sigma}_s((X_{s}^{i,N},Y_{s}^{i,N}),(X_{s}^{m,N},Y_{s}^{m,N})) \frac{1}{N}\sum_{m=1}^N\tilde{\sigma}_s^\ast((X_{s}^{i,N},Y_{s}^{i,N}),(X_{s}^{m,N},Y_{s}^{m,N}))\\
&-\int_{\R^{2d}}\tilde{\sigma}_s((X_{s}^{i,N},Y_{s}^{i,N}),y)\mu_s(\d y)\int_{\R^{2d}}\tilde{\sigma}_s^\ast((X_{s}^{i,N},Y_{s}^{i,N}),y)\mu_s(\d y).
\end{align*}
Therefore, we conclude that
\begin{align}\label{DUH1}
&\E F((X_t^{i,N},Y_t^{i,N})_{1\leq i\leq N})-\E F((\bar{X}_t^{i},\bar{Y}_t^{i})_{1\leq i\leq N})=R_{0,t}^{b}F+R_{0,t}^{\sigma}F, \ \ F\in C_b^2((\R^{2d})^N)
\end{align}
    with
\begin{align*}
 R_{0,t}^{b}F
  &=\int_0^t \E\sum_{i=1}^N\Bigg\<B^i_s,[\nabla_{i,y}(P^\mu_{s,t})^{\otimes N}F]((X_{s}^{m,N}, Y_{s}^{m,N})_{1\leq m\leq N})\Bigg\>\d s,
   \end{align*}
   and
 \begin{align*}
R_{0,t}^{\sigma}F
    &=\frac{1}{2}\int_0^t \E\sum_{i=1}^N\mathrm{tr}\Bigg \{\Sigma^i_s[\nabla^2_{i,y}(P^\mu_{s,t})^{\otimes N}F]((X_{s}^{m,N}, Y_{s}^{m,N})_{1\leq m\leq N})\Bigg\}\d s.
\end{align*}
Let $\theta_{s,t}(z)=(\theta_{s,t}^1(z),\theta_{s,t}^2(z))$ solve
$$\begin{cases}
\ff{\d}{\d t} \theta_{s,t}^1(z)=\theta_{s,t}^2(z), \\
\ff{\d}{\d t} \theta_{s,t}^2(z)=b_t(\theta_{s,t}(z),\mu_t), \ \ t\in [s,T], \theta_{s,s}(z)=z\in\R^{2d}.
\end{cases}
$$
Again by  It\^o's formula, we find a constant $c_1>0$ and a martingale $\tilde{M}_t$  such that
\beg{align*} &\d |(X_{s,t}^{\mu,z},Y_{s,t}^{\mu,z})-\theta_{s,t}(z)|^2\\
 &= 2 \big\<X_{s,t}^{\mu,z}-\theta_{s,t}^1(z),  Y_{s,t}^{\mu,z} -\theta_{s,t}^2(z)\big\> \d t\\
&+2 \big\<Y_{s,t}^{\mu,z}-\theta_{s,t}^2(z),  b_t((X_{s,t}^{\mu,z}, Y_{s,t}^{\mu,z}),\mu_t) -b_t(\theta_{s,t}(z),\mu_t)\big\> \d t\\
&+ \| \si_t ((X_{s,t}^{\mu,z},X_{s,t}^{\mu,z}),\mu_t)\|_{HS}^2 \d t + \d \tilde{M}_t\\
&\le c_1 \big\{|(X_{s,t}^{\mu,z},Y_{s,t}^{\mu,z})-\theta_{s,t}(z)|^2 +1\Big\}\d t+\d \tilde{M}_t,\ \ \ t\in [s,T], |(X_{s,s}^{\mu,z},Y_{s,s}^{\mu,z})-\theta_{s,s}(z)|=0.\end{align*}
So, we have
$$\E\big[|(X_{s,t}^{\mu,z},Y_{s,t}^{\mu,z})-\theta_{s,t}(z)|^2\big]\le c_1 \e^{c_1 T}(t-s),\ \ \ 0\le s\le t\le T.$$
Combining this with \eqref{gre1}-\eqref{gre2}, we derive for any $f\in C_b^2(\R^{2d})$ with $[f]_\eta<1$,
\begin{align}\label{gra1}
\nonumber|\nabla^{j}_yP_{s,t}^\mu f|(z)&=|\nabla^{j}_yP_{s,t}^\mu (f-f(\theta_{s,t}(z)))|(z)\leq c(t-s)^{-\frac{j}{2}}\{P_{s,t}^\mu |f-f(\theta_{s,t}(z))|^2(z)\}^{\frac{1}{2}}\\
&\leq c(t-s)^{-\frac{j}{2}}\{\E|(X_{s,t}^{\mu,z},Y_{s,t}^{\mu,z})-\theta_{s,t}(z)|^2\}^{\frac{\eta}{2}}\\
\nonumber&\leq c_0(t-s)^{\frac{-j+\eta}{2}},\ \ 0\leq s<t\leq T,j=1,2,z\in\R^{2d},
\end{align}
and
\begin{align}\label{gre3}
|\nabla P_{s,t}^\mu f|(z)\leq c_0(t-s)^{-\frac{3+\eta}{2}},\ \ 0\leq s<t\leq T,z\in\R^{2d}
\end{align}
for some constant $c_0>0$.
With \eqref{S12}, \eqref{gre1}-\eqref{gre3} in hand, by repeating the proof of Theorem \ref{POC10}(2), we derive \eqref{CMY1} and the proof is completed.
\end{proof}
\begin{rem} From the proofs of Theorem \ref{POC10} and Theorem \ref{POC12}, if $\sigma$ only depends on the time and spatial variables, then the term $R_{0,t}^{\sigma}F$ in \eqref{DUH} and \eqref{DUH1} vanishes. In this case, $\eta$ can be equal to $0$ and the quantitative propagation of chaos in total variation distance may be derived.
\end{rem}

\end{document}